\documentclass[12pt,psamsfonts]{amsart}
\usepackage{amsmath}
\usepackage{amsthm}
\usepackage{amssymb}
\usepackage{amscd}
\usepackage{amsfonts}
\usepackage{amsbsy}
\usepackage{color}
\usepackage{graphicx}
\usepackage{graphics}

\newcommand{\R}{\ensuremath{\mathbb{R}}}
\newcommand{\T}{\ensuremath{\mathbb{T}}}
\newcommand{\Cc}{\ensuremath{\mathbb{C}}}

\newcommand{\X}{\mathcal{X}}
\newcommand{\Y}{\mathcal{Y}}

\def\p{\partial}
\def\e{\varepsilon}

\newtheorem {theorem} {Theorem}
\newtheorem {proposition} [theorem]{Proposition}

\newtheorem {lemma}  [theorem]{Lemma}

\begin{document}

\title[Polynomial vector fields on the Clifford torus]
{Polynomial vector fields\\ on the Clifford torus}

\author[J. Llibre and A.C. Murza]
{Jaume Llibre and Adrian C. Murza}
\address{Jaume Llibre, Departament de Matem\`atiques, Universitat
Aut\`onoma de Barcelona, 08193 Bellaterra, Barcelona, Catalonia,
Spain} \email{jllibre@mat.uab.cat}

\address{Adrian C. Murza, Institute of Mathematics ``Simion Stoilow''
of the Romanian Academy, Calea Grivi\c tei 21, 010702 Bucharest,
Romania} \email{adrian\_murza@hotmail.com}


\begin{abstract}
First we characterize all the polynomial vector fields in $\R^4$
which have the Clifford torus as an invariant surface. After we
study the number of invariant meridians and parallels that such
polynomial vector fields can have in function of the degree of these
vector fields.
\end{abstract}

\maketitle

\section{Introduction and statement of the main results}

The Clifford torus
\begin{equation*}
\T=\left\{(x_1,x_2,x_3,x_4)\in\R^4:\,\, x_1^2+x_2^2=\frac{1}{2},\,\,
x_3^2+x_4^2=\frac{1}{2}\right\}
\end{equation*}
in geometric topology is the simplest and most symmetric Euclidean
space embedding of the cartesian product of two circles. It lives in
$\R^4$, as opposed to $\R^3$.

In MathSciNet at July 22 of 2017 it appears with the keyword
``Clifford torus'' 430 references. The more recent reference is
\cite{LW}. In the reference \cite{DT} are studied the meridians of
the surfaces of revolution and some information about the meridians
of the Clifford torus can be found there. In the references
\cite{BBS, ZZ} are studied the parallels of the surfaces of
revolutions and again contains some information on the parallels of
the Clifford torus.

In this paper first we shall study the polynomial vector fields of
arbitrary degree in $\R^4$ having the Clifford torus invariant by
their flow, and after we shall compute the maximal number of
parallels and meridians that a polynomial vector field of a given
degree can exhibit on the Clifford torus.

The maximum number of invariant hyperplanes that a polynomial vector
field in $\R^n$ can have in function of its degree was given in
\cite{LM}. The analogous result for the invariant straight lines of
polynomial vector fields in $\R^2$ was provided before in
\cite{AGL}. The study of the maximum number of meridians and
parallels for a torus in $\R^3$ were studied in \cite{LM1}, and for
an algebraic torus in \cite{LR}. In surfaces of revolution in $\R^3$
the meridians and parallels invariant by polynomial vector fields
have been studied in \cite{DLM}.

As usual we denote by $\R[x_1,x_2,x_3,x_4]$ the ring of the
polynomials in the variables $x_1,x_2,x_3$ and $x_4$ with real
coefficients. By definition a {\it polynomial differential system
in} $\R^4$ is a system of the form
\begin{equation}\label{e1}
\frac{dx_i}{dt}=P_i(x_1,x_2,x_3,x_4), \quad \mbox{for $i=1,2,3,4$},
\end{equation}
where $P_i(x_1,x_2,x_3,x_4)\in\R[x_1,x_2,x_3,x_4]$. If $m_i$ is the
degree of the polynomial $P_i$, then $m =
\mathrm{max}\{m_1,m_2,m_3,m_4\}$ is the degree of the polynomial
differential system \eqref{e1}.

We denote by
\begin{equation}\label{e2}
\X = \sum_{i=1}^4 P_i(x_1,x_2,x_3,x_4)\frac{\partial}{\partial x_i},
\end{equation}
the {\it polynomial vector field} associated to the differential
polynomial system \eqref{e1} of degree $m.$

An {\it invariant algebraic hypersurface} for the polynomial
differential system \eqref{e1} or for the polynomial vector field
\eqref{e2} is an algebraic surface $f=f(x_1,x_2,x_3,x_4) = 0$ with
$f\in\R[x_1,x_2,x_3,x_4],$ such that for some polynomial
$K\in\R[x_1,x_2,x_3,x_4]$ we have
\begin{equation}\label{e21}
\X f = Kf.
\end{equation}
Therefore if a solution curve of system \eqref{e1} has a point on
the algebraic hypersurface $f = 0,$ then the whole solution curve is
contained in $f = 0.$ The polynomial $K$ is called the {\it
cofactor} of the invariant algebraic hypersurface $f = 0.$ We remark
that if the polynomial system has degree $m,$ then any cofactor has
at most degree $m-1.$

If $f=0$ and $g=0$ are two invariant algebraic hypersurfaces by the
polynomial vector field $\X$, then $S=\{f=0\}\cap \{g=0\}$ is an
{\it invariant algebraic surface} by the vector field $\X$, i.e. if
an orbit of $\X$ has a point on the algebraic surface $S$, then the
whole orbit is contained in $S$.

In the next theorem we characterize all the polynomial differential
systems having the Clifford torus $\T$ as an invariant algebraic
surface.

\begin{theorem}\label{t1}
If the polynomial differential system \eqref{e1} has the Clifford
torus $\T$ invariant, then
\begin{equation}\label{e3}
\begin{array}{l}
P_1= A(x_1^2+x_2^2-1/2)-2Cx_2, \vspace{0.2cm}\\
P_2= B(x_1^2+x_2^2-1/2)+2Cx_1, \vspace{0.2cm}\\
P_3= D(x_3^2+x_4^2-1/2)-2Fx_4, \vspace{0.2cm}\\
P_4= E(x_3^2+x_4^2-1/2)+2Fx_3,
\end{array}
\end{equation}
where $A$, $B$, $C$, $D$, $E$ and $F$ are arbitrary polynomials in
the variables $x_1$, $x_2$, $x_3$ and $x_4$.
\end{theorem}

Theorem \ref{t1} is proved in section \ref{s2}.

For all $(a,b)\in \R^2$ such that $a^2+b^2=1/2$ a {\it meridian} of
the Clifford torus $\T$ is
$$
M_{(a,b)}= \{(x_1,x_2,a,b)\in \R^4 : \, x_1^2+x_2^2=1/2\},
$$
and a {\it parallel} is
$$
P_{(a,b)}= \{(a,b,x_3,x_4)\in \R^4 : \,x_3^2+x_4^2=1/2\}.
$$
So a meridian is defined by the three equations $x_1^2+x_2^2=1/2$,
$x_3=a$, $x_4=b$ with $a^2+b^2=1/2$, and similarly by a parallel.

In the next theorem we provide the maximum number of invariant
meridians or parallels that a polynomial vector field $\X$ on $\T$
can have in function of its degree. See section \ref{s3} for the
definition of the multiplicity of a meridian and of a parallel.

\begin{theorem}\label{t2}
Let $\X$ be a polynomial vector field on the Clifford torus $\T$ of
degree ${\bf m}=(m_1,m_2,m_3,m_{4})$ with $m_1\geqslant m_2\geqslant
m_3\geqslant m_{4}>0.$
\begin{itemize}
\item[(a)] The number of invariant meridians of $\mathcal{X}$ is at most
$m_4-2$ taking into account their multiplicities if $m_4>3$, and $4$
if $m_4=3$. These upper bounds are reached.

\smallskip

\item[(b)] The number of invariant parallels of $\mathcal{X}$ is at most
$m_2-2$ taking into account their multiplicities if $m_2>3$, and $4$
if $m_2=3$. These upper bounds are reached.
\end{itemize}
\end{theorem}

Theorem \ref{t2} is proved in section \ref{s3}.

\section{Proof of Theorem \ref{t1}}\label{s2}

For proving Theorem \ref{t1} we shall need some definitions and
results for the polynomial differential systems in $\R^2$.

Consider the following polynomial differential system in $\R^2$
\begin{equation}\label{e4}
\begin{array}{l}
\dot x=P(x,y), \vspace{0.2cm}\\
\dot y=Q(x,y),
\end{array}
\end{equation}
and let
$$
\Y= P(x,y)\frac{\p}{\p x}+ Q(x,y)\frac{\p}{\p y},
$$
its associated polynomial vector field.

Let $f(x,y)$ a polynomial. Then the algebraic curve $f(x,y)=0$ is
{\it invariant} by system \eqref{e4} if there exists a polynomial
$k=k(x,y)$ such that $\Y f=kf$.

The next result is proved in Lemma 6 of \cite{CLPZ}.

\begin{lemma}\label{L1}
Assume that the polynomial system \eqref{e4} has an invariant
algebraic curve $f(x,y)=0$ without singularities (i.e. there are no
points at which $f$ and its first derivatives are all vanish). If
$(f_x,f_y)=1$ (i.e. the polynomials $f_x$ and $f_y$ has no common
factors), then
\begin{equation*}
\begin{array}{l}
P= Af-Cf_y,\vspace{0.2cm}\\
Q=Bf+Df_x,
\end{array}
\end{equation*}
where $A,B$ and $C$ are arbitrary polynomials in the variables $x$
and $y$.
\end{lemma}

\begin{proof}[Proof of Theorem \ref{t1}]
We consider polynomial vector fields $\X$ given in \eqref{e2} of
degree $m$ in $\R^4$ having the Clifford torus $\T$ as an {\it
invariant algebraic surface}, i.e. both hypersurfaces
$x_1^2+x_2^2=1/2$ and $x_3^2+x_4^2=1/2$ are invariant by $\X$.

Let $f= x_1^2+x_2^2-1/2=0$ and $g= x_3^2+x_4^2-1/2=0$. By Lemma
\ref{L1} and from the definition of the invariant algebraic
hypersurface $f=0$ given in \eqref{e21} it follows that
\begin{equation}\label{e5}
\begin{array}{l}
P_1= A(x_1^2+x_2^2-1/2)-2Cx_2, \vspace{0.2cm}\\
P_2= B(x_1^2+x_2^2-1/2)+2Cx_1,
\end{array}
\end{equation}
where $A$, $B$ and $C$ arbitrary polynomials in the variables $x_1$,
$x_2$, $x_3$ and $x_4$. In a similar way, from Lemma \ref{L1} and
the definition of the invariant algebraic hypersurface $g=0$ we get
\begin{equation}\label{e6}
\begin{array}{l}
P_3= D(x_3^2+x_4^2-1/2)-2Fx_4, \vspace{0.2cm}\\
P_4= E(x_3^2+x_4^2-1/2)+2Fx_3,
\end{array}
\end{equation}
where $D$, $E$ and $F$ arbitrary polynomials in the variables $x_1$,
$x_2$, $x_3$ and $x_4$. In short from \eqref{e5} and \eqref{e6} the
theorem follows.
\end{proof}

\section{Proof of Theorem \ref{t2}}\label{s3}

One of the best tools for working with invariant algebraic
hypersurfaces is the {\it extactic polynomial of $\X$ associated to
a finite vector space of polynomials generated by $W$}. To our
knowledge the extactic polynomial was introduced by Lagutinskii, see
\cite{SW}. We recall its definition for a polynomial vector field in
$\R^4$. Let $W$ be a finitely generated vector subspace of the
vector space $\Cc[x_1,x_2,x_3,x_d]$ generated by the basis
$\{v_{1}$, \ldots, $v_{l}\}$. The {\it extactic polynomial of $\X$
associated to $W$} is
\begin{equation*}
{\mathcal E}_W(\X)={\mathcal E}_{\{v_{1}, \ldots, v_{l}\}}(\X) =
\det\left(
\begin{array}{cccc}
v_{1} & v_{2} & \ldots & v_{l} \\
\X(v_{1}) & \X(v_{2}) & \ldots & \X(v_{l}) \\
\vdots & \vdots & \ldots & \vdots \\
\X^{l-1}(v_{1}) & \X^{l-1}(v_{2}) & \ldots & \X^{l-1}(v_{l})
\end{array}
\right),
\end{equation*}
where $\X^{j}(v_{i})= \X^{j-1}(\X(v_{i}))$. The extactic polynomial
does not dependent of the chosen basis of $W$.

The extactic polynomial ${\mathcal E}_{W}(\X)$ has two good
properties. First, it allows to detect invarian algebraic
hypersurfaces $f=0$ with $f\in W$ by the polynomial vector field
$\X$, see the following proposition proved in \cite{CLP}. Second, it
allows to compute the multiplicity of the invariant algebraic
hypersurfaces.

Even if the next proposition is stated for complex polynomial vector
fields, it is very useful for our later considerations. This is so,
because we deal with real polynomial vector fields, which are
particular cases of complex ones.

\begin{proposition}\label{p1}
Let $\X$ be a polynomial vector field in $\Cc^4$ and let $W$ be a
finitely generated vector subspace of $\Cc[x_1,x_2,x_3,x_4]$ with
$\dim (W)>1 $. Then every algebraic invariant hypersurface $f=0$ for
the vector field $\X$, with $f\in W$, is a factor of the polynomial
${\mathcal E}_{W}(\X)$.
\end{proposition}

{From} Proposition \ref{p1} it follows that $f=0$ is an invariant
hyperplane of the polynomial vector field $\X$ if the polynomial $f$
is a factor of the polynomial ${\mathcal E}_{W}(\X)$, where $W$ is
generated by $\{1,x_1,x_2,x_3,x_4\}$.

{From} \cite{CLP} the invariant hypersurface $f=0$, with $f\in W$,
has {\it multiplicity} $k$ if $k$ is the greatest positive integer
such that $f^k$ divides the polynomial ${\mathcal E}_{W}(\X)$. In
\cite{CLP} it is proved that if we have that $f=0$ is an invariante
hypersurface of multiplicity $k$, then in a neighborhood of $\X$ in
the topology of the coefficients there are polynomial vector fields
$\Y_\varepsilon$, being $\e$ a small parameter, having $k$ invariant
algebraic hypersurfaces such that all of them tend to the
hypersurface $f=0$ when $\e\to 0$.

We say that the meridian $M_{(a_i,b_i)}$ with $a_i^2+b_i^2=1/2$ has
{\it multiplicity} $k$ if both invariant hyperplanes $x_3-a_i=0$ and
$x_4-b_i=0$ of the differential system \eqref{e1} with the
polynomials $P_i$ given by \eqref{e3} have multiplicities $k_1$ and
$k_2$ respectively, and $\min\{k_1,k_2\}=k$. In a similar way is
defined the multiplicity of a parallel $P_{(a_i,b_i)}$.

\begin{proof}[Proof of Theorem \ref{t2}]
A meridian of the Clifford torus $\T$ is obtained by intersecting
$\T$ with the hyperplanes $x_3=a$ and $x_4=b$ with $a^2+b^2=1/2$. So
the hyperplanes $x_3-a=0$ and $x_4-b=0$ must be invariant by the
polynomial vector field $\X$. In other words $x_3-a$ must divide the
extactic polynomial
\begin{equation*}
{\mathcal E}_{1,x_3}(\X)= \left|
\begin{array}{cccc}
1 & x_3 \\
0 & P_3
\end{array}
\right|=P_3,
\end{equation*}
i.e. $x_3-a$ must divide the polynomial $P_3(x_1,x_2,x_3,x_4)$. In a
similar way $x_4-b$ must divide the polynomial $P_4(x_1,x_2,x_3,
x_4)$. Since the degrees of $P_i$ is $m_i$ for $i=3,4$, it follows
that the polynomials $x_3-a_i$ at most divide $m_3$ times the
polynomial $P_3$. This is only possible if
\begin{equation}\label{f1}
P_3=k_3\prod_{i=1}^{m_3} (x_3-a_i), \quad \mbox{and} \quad
P_4=k_4\prod_{i=1}^{m_4} (x_4-b_i),
\end{equation}
with $k_3,k_4\in \R\setminus \{0\}$.  But taking into account the
expressions of $P_3$ and $P_4$ given in \eqref{e3} we only can
obtain the expressions of \eqref{f1} if and only if
\begin{equation}\label{f2}
D=2k x_3, \quad E=-2k x_4, \quad \mbox{and} \quad F=k x_3 x_4,
\end{equation}
with $k\in \R\setminus \{0\}$. Then
\[
\begin{array}{l}
P_3= 2k x_3 ( x_3^2 + x_4^2-1/2) -2k x_3 x_4^2= 2k x_3 (x_3^2-1/2),
\vspace{0.2cm}\\
P_4= -2k x_4 ( x_3^2 + x_4^2-1/2) +2k x_3^2 x_4= 2k x_4 (x_4^2-1/2).
\end{array}
\]
So in this case $m_3=m_4=3$ and we have $4$ meridians, namely
$M_{(0,1/\sqrt{2})}$, $M_{(0,-1/\sqrt{2})}$, $M_{(1/\sqrt{2},0)}$
and $M_{(-1/\sqrt{2},0)}$.

Expect this exceptional case \eqref{f2} that we can obtain for the
polynomials $P_3$ and $P_4$ the expressions \eqref{f1}, we have that
there are at most $m_3-2$ invariant hyperplanes of the form
$x_3-b_i$ taking into account their multiplicities, and at most we
have $m_4-2$ invariant hyperplanes of the form $x_4-b_i$ taking into
account their multiplicities. These invariants hyperplanes are
obtained when
\[
D=k_3\prod_{i=1}^{m_3-2}(x_3-a_i),\quad
E=k_4\prod_{i^1}^{m_4-2}(x_4-b_i), \quad F=0,
\]
with $k_3,k_4\in \R\setminus \{0\}$. Therefore the differential
system \eqref{e1} with the polynomials $P_i$ given by \eqref{e3} has
the invariant meridians $M_{(a_i,b_i)}$ if $a_i^2+b_i^2=1/2$ for
$i=1,\ldots,m_4-2$. Eventually some of these invariant meridians can
have multiplicity larger than one if the $a_i$ and $b_i$ appears
repeated in the expressions of the polynomials $D$ and $E$.

In short, if $m_4>3$ an upper bound for the maximum number of
invariant meridians is $m_4-2$ taking into account their
multiplicities, because $m_3\ge m_4$; or if $m_4=3$ that upper bound
is $4$. Note from this proof that these upper bounds are reached.
This proves statement $(a)$ of Theorem \ref{t2}.

In an analogous way we obtain that the maximum number of invariant
parallels is $m_2-2$ taking into account their multiplicities if
$m_2>3$, or $4$ if $m_2=3$. This proves statement $(b)$ of Theorem
\ref{t2}.
\end{proof}

\section*{Acknowledgements}

The first author is partially supported by a FEDER-MINECO grant
MTM2016-77278-P, a MINECO grant MTM2013-40998-P, and an AGAUR grant
2014SGR-568. The second author acknowledges support from a grant of
the Romanian National Authority for Scientific Research and
Innovation, CNCS-UEFISCDI, project number PN-II-RU-TE-2014-4-0657.

\end{document}